\documentclass{article}

\usepackage[latin5]{inputenc}
\usepackage{amsfonts}
\usepackage{amsmath}
\usepackage{xypic}

\newenvironment{proof}[1][Proof]{\textbf{#1 }}{$\square$}
\newtheorem{theorem}{Theorem}

\newtheorem{definition}[theorem]{Definition}

\newtheorem{proposition}[theorem]{Proposition}

\title{Spin$^{T}$ structure and Dirac operator on Riemannian manifolds}

\usepackage{fancyhdr}
\begin{document}

\date{}
\maketitle
\begin{center}
\c{S}enay BULUT \\[0pt]
(Department of Mathematics, Anadolu University, Eskisehir, TURKEY)\\[0pt]
\texttt{skarapazar@anadolu.edu.tr}

Ali Kemal ERKOCA\\[0pt]
(Department of Mathematics, Anadolu University, Eskisehir, TURKEY)\\[0pt]
\texttt{ake@anadolu.edu.tr}\\[0pt]

\end{center}

\begin{abstract}
In this paper, we describe the group Spin$^T(n)$ and give some properties of this group. We construct Spin$^T$ spinor bundle $\mathbb{S}$ by means of the spinor representation of the group Spin$^T(n)$ and define covariant derivative operator and Dirac operator on $\mathbb{S}$. Finally, Schrödinger-Lichnerowicz-type formula is derived by using these operators.

\end{abstract}

\textbf{Key Words} Spinor bundle, the group Spin$^T(n)$, Dirac operator, Schrödinger-Lichnerowicz-type formula.\newline
\textbf{2000 MR Subject Classification} 15A66, 58Jxx. \label{first}

\section{Introduction}
Spin and Spin$^c$ structures is effective tool to study the geometry and topology of manifolds, especially in dimension four. Spin and Spin$^c$ manifolds have been studied extensively in \cite{F,L,S,N}. For any compact Lie group $G$ the Spin$^G$ structure have been studied in \cite{T}. However, the spinor representation is replaced by a hyperkahler manifold, also called target manifold.  In this paper, we define the Lie group Spin$^T(n)$ as a quotient group by taking $G=S^1\times S^1$. The groups Spin$(n)$ and Spin$^c(n)$ are the subset of Spin$^T(n)$. We define Spin$^T$ structure on any Riemannian manifold. The spinor representation of Spin$^T(n)$ is defined by the help of the spinor representation of Spin$(n)$. By using the spinor representation of Spin$^T(n)$ we construct the Spin$^T$ spinor bundle $\mathbb{S}$. Finally, we give Schrödinger-Lichnerowicz-type formula by using covariant derivative operator and Dirac operator on $\mathbb{S}$.

This paper is organized as follows. We begin with a section introducing the group Spin$^T(n)$. In the following section, we define Spin$^T$ structure on any Riemannian manifold. The final section is dedicated to the construction of the spinor bundle $\mathbb{S}$, the study of the Dirac operator associated to Levi-Civita connection $\nabla$ and Schrödinger-Lichnerowicz-type formula.

\section{The group Spin$^T(n)$}
\begin{definition}
The Spin$^T$ group is defined as 
$$Spin^{T}(n):=(Spin(n)\times S^{1}\times S^{1})/\{\pm1\}.$$
\end{definition}
The elements of Spin$^{T}(n)$ are thus classes $[g,z_{1},z_{2}]$ of pairs $(g,z_{1},z_{2})\in Spin(n)\times S^{1}\times S^{1}$ under the equivalence relation \\
$$(g,z_{1},z_{2})\sim(-g,-z_{1},-z_{2}).$$ \\
We can define the following homomorphisms:
\begin{enumerate}
\item[a.] The map $\lambda^{T}:Spin^{T}(n)\longrightarrow SO(n)$ is given by $\lambda^{T}([g,z_{1},z_{2}])=\lambda(g)$ where the map $\lambda:Spin(n)\rightarrow SO(n)$ is the two-fold covering given by $\lambda(g)(v)=gvg^{-1}$.
\item[b.] $i:Spin(n)\longrightarrow Spin^{T}(n)$ is the natural inclusion map $i(g)=[g,1,1]$. 
\item[c.] $j:S^{1}\times S^{1}\longrightarrow Spin^{T}(n)$ is the inclusion map $j(z_{1},z_{2})=[1,z_{1},z_{2}]$.
\item[d.] $l:Spin^{T}(n)\longrightarrow S^{1}\times S^{1}$ is given by $l([g,z_{1},z_{2}])=(z_{1}^2,z_{1}z_{2})$.
\item[e.] $p:Spin^{T}(n)\longrightarrow SO(n)\times S^{1}\times S^{1}$ is given by $p([g,z_{1},z_{2}])=(\lambda(g),z_{1}^2,z_{1}z_{2})$. Hence, $p=\lambda^{T}\times l$. Here $p$ is a $2$-fold covering.
\end{enumerate}
Thus, we obtain the following commutative diagram where the row and the column are exact.
$$\xymatrix{
         &                    &1       \ar[d]          & \\
         &                    &S^{1}\times S^{1} \ar[d]^{j} \ar[rd] & \\
1 \ar[r] & Spin(n) \ar[r]^{i} \ar[rd]^{\lambda} & Spin^{T}(n) \ar[r]^{l} \ar[d]^{\lambda^{T}}  & S^{1}\times S^{1} \ar[r] & 1\\
         &                    &SO(n)   \ar[d]  &           \\
         &                    &1                                                                          }$$
Moreover, we have the following exact sequence:
$$1\longrightarrow \mathbb{Z}_2\longrightarrow Spin^{T}(n)\overset{p}{\longrightarrow} SO(n)\times S^1\times S^1\longrightarrow 1.$$

\begin{theorem}
The group Spin$^T(n)$ is isomorphic to Spin$^c(n)\times S^1$.
\end{theorem}
\begin{proof}
We define the map $\varphi$ in the following way:
\begin{equation*}
\begin{array}{lll}
    Spin(n)\times S^1\times S^1&\overset{\varphi}{\rightarrow}  & Spin^c(n)\times S^1   \\
     (g,z_1,z_2) & \mapsto &  ([g,z_1],z_1z_2) 
\end{array}
\end{equation*}
It can be easily shown that $\varphi$ is a surjective homomorphism and the kernel of $\varphi$ is $\{(1,1,1),(-1,-1,-1)\}$. Thus, the group Spin$^T(n)$ is isomorphic to Spin$^c(n)\times S^1$.
\end{proof}

Since Spin$(n)$ is contained in the complex Clifford algebra $\mathbb{C}l_{n}$, the spin representation $\kappa$ of the group Spin$(n)$ extends to a Spin$^{T}(n)$-representation. For an element $[g,z_{1},z_{2}]$ from Spin$^{T}(n)$ and any spinor $\psi\in\Delta_{n}$, the spinor representation $\kappa^{T}$ of Spin$^{T}(n)$ is given by $$\kappa^{T}[g,z_{1},z_{2}]\psi=z_{1}^{2}z_{2}\kappa(g)(\psi).$$

\begin{proposition}
If $n=2k+1$ is odd, then $\kappa^{T}$ is irreducible.
\end{proposition}
\begin{proof}
Assume that $\{0\}\neq W\neq \Delta_{2k+1}$ is a Spin$^T$ invariant subspace. Thus, we have $\kappa^{T}[g,z_{1},z_{2}](W)\subseteq W$. That is, $z_{1}^{2}z_{2}\kappa(g)(W)\subseteq W$. In this case, for every $w\in W$ there exists a $w'\in W$ such that $z_{1}^{2}z_{2}\kappa(g)(w)= w'$. As $\kappa(g)(w)=\displaystyle\frac{1}{z_{1}^{2}z_{2}}w'\in W$ and the representation $\kappa$ of Spin$(n)$ is irreducible if $n$ is odd, this is a contradiction. The representation $\kappa^{T}$ of Spin$^{T}(n)$ has to be irreducible for $n=2k+1$.

\end{proof}

\begin{proposition}
If $n=2k$ is even, then the spinor space $\Delta_{2k}$ decomposes into two subspaces $\Delta_{2k}=\Delta_{2k}^+\oplus \Delta_{2k}^-$.
\end{proposition}
\begin{proof}
We know that the Spin$(n)$ representation $\Delta_{2k}$ decomposes into two subspaces $\Delta_{2k}^+$ and $\Delta_{2k}^-$. Thus, we obtain $z_{1}^{2}z_{2}\kappa(g)(\Delta_{2k}^+)\subseteq \Delta_{2k}^+$ and $z_{1}^{2}z_{2}\kappa(g)(\Delta_{2k}^-)\subseteq \Delta_{2k}^-$. Namely, $\kappa^{T}[g,z_{1},z_{2}](\Delta_{2k}^+)\subseteq \Delta_{2k}^+$ and $\kappa^{T}[g,z_{1},z_{2}](\Delta_{2k}^-)\subseteq \Delta_{2k}^-$. Hence, the Spin$^T(2k)$ representation $\Delta_{2k}$ decomposes into two subspaces $\Delta_{2k}^+$ and $\Delta_{2k}^-$.
It can be easily seen that the Spin$^T(2k)$ representation $\Delta_{2k}^{\pm}$ is irreducible.
\end{proof}

The Lie algebra of the group Spin$^{T}(n)$ is described by $$\mathfrak{spin}^{T}(n)=\mathfrak{m}_2\oplus i\mathbb{R} \oplus i\mathbb{R}.$$ The differential $p_*:\mathfrak{spin}^{T}(n)\rightarrow \mathfrak{so}(n)\oplus  i\mathbb{R}\oplus i\mathbb{R}$ is defined by 
$$p_*(e_{\alpha}e_{\beta},\lambda i,\mu i)=(2E_{\alpha\beta},2\lambda i,(\lambda+\mu)i)$$ where $\lambda$ and $\mu$ are any real numbers and $E_{\alpha\beta}$ is the $n\times n$ matrix with entries $(E_{\alpha\beta})_{\alpha\beta}=-1$, $(E_{\alpha\beta})_{\beta\alpha}=1$ and all others are equal to zero.
The inverse of the differential $p_*$ is given by $$p_*^{-1}(E_{\alpha\beta},\lambda i,\mu i)=(\displaystyle\frac{1}{2}e_{\alpha}e_{\beta},\displaystyle\frac{1}{2}\lambda i,(\mu-\frac{1}{2}\lambda) i).$$

\section{Spin$^{T}$ structure}
\begin{definition}
A Spin$^{T}$ structure on an oriented Riemannian manifold $(M^n,g)$ is a Spin$^{T}(n)$ principal bundle $P_{Spin^T(n)}$ together with a smooth map \\ $\Lambda:P_{Spin^T(n)}\rightarrow P_{SO(n)}$ such that the following diagram commutes:
\begin{equation*}
\xymatrix{\,P_{Spin^T(n)}\times Spin^T(n) \ar[r]^{\hspace{.1cm}}
\ar[d]^{\Lambda\times\lambda^T } &
P_{Spin^T(n)} \ar[d]^{\Lambda}\\
P_{SO(n)}\times SO(n) \ar[r] &
P_{SO(n)}}
\end{equation*}

\end{definition}
From above definition we can construct a two-fold covering map $$\Pi:P_{Spin^T(n)}\rightarrow P_{SO(n)}\times P_{S^1\times S^1}.$$
Given a Spin$^T$ structure $(P_{Spin^T(n)},\Lambda)$, the map $\lambda^{T}:Spin^{T}(n)\longrightarrow SO(n)$ induces an isomorphism $$P_{Spin^T(n)}/{S^1\times S^1}\cong P_{SO(n)}.$$ In similar way, $Spin^T(n)/_{Spin(n)}\cong S^1\times S^1$ implies the isomorphism $$P_{Spin^T(n)}/{Spin(n)}\cong P_{S^1\times S^1}.$$
Note that on account of the inclusion map $i:Spin(n)\rightarrow Spin^T(n)$, every spin structure on $M$ induces a Spin$^T$ structure. Similarly, since there exists a inclusion map $Spin^c(n)\rightarrow Spin^T(n)$, every Spin$^c$ structure on $M$ induces a Spin$^T$ structure.

\section{Spinor bundle and Dirac operator}
Let $(M^n,g)$ be an oriented connected Riemannian manifold and $P_{SO(n)}\rightarrow M$ the $SO(n)-$principal bundle of positively oriented orthonormal frames. The Levi-Civita connection $\nabla$ on $P_{SO(n)}$ determine a connection $1-$form $\omega$ on the principal bundle $P_{SO(n)}$ with values in $\mathfrak{so}(n)$, locally given by
\begin{equation*}
\omega^e=\sum_{i<j}g(\nabla e_i,e_j)E_{ij}
\end{equation*}
where $e=\{e_1,\ldots,e_n\}$ is a local section of $P_{SO(n)}$ and $E_{ij}$ is the $n\times n$ matrix with entries $(E_{ij})_{ij}=-1$, $(E_{ij})_{ji}=1$ and all others are equal to zero.

We fix a connection $$(A,B):TP_{S^1\times S^1}\rightarrow i\mathbb{R}\oplus i\mathbb{R}$$ on the principal bundle $P_{S^1\times S^1}.$ The connections $\omega$ and $(A,B)$ induce a connection $$\omega\times (A,B):T(P_{SO(n)}\times P_{S^1\times S^1})\rightarrow \mathfrak{so}(n)\oplus i\mathbb{R}\oplus i\mathbb{R}$$ on the fibre product bundle $P_{SO(n)}\times P_{S^1\times S^1}$. Now we can define a connection $1-$form $\widetilde{\omega\times (A,B)}$ on the principal bundle $P_{Spin^T(n)}$ such that the following diagram commutes: 
\begin{equation*}
\xymatrix{\,TP_{Spin^T(n)} \ar[r]^{\hspace{.1cm}}
\ar[d]^{\Pi_* }  \ar[r]^{\widetilde{\omega\times (A,B)}\hspace{1cm}}&
\mathfrak{spin}^T(n)=\mathfrak{m}_2\oplus i\mathbb{R}\oplus i\mathbb{R} \ar[d]^{p_*}\\
T(P_{SO(n)}\times P_{S^1\times S^1}) \ar[r]  \ar[r]^{\omega\times (A,B)\hspace{-.8cm}}  &
\mathfrak{so}(n)\oplus i\mathbb{R} \oplus i\mathbb{R}  }
\end{equation*}
That is, the equality $$p_*\circ \widetilde{\omega\times (A,B)}=\omega\times (A,B)\circ \Pi_*$$ holds.
\begin{definition}
The spinor bundle of a Spin$^T$ manifold is defined as the associated vector bundle $$\mathbb{S}=P_{Spin^T(n)}\times_{\kappa^T}\Delta_n$$where $\kappa^T:Spin^T(n)\rightarrow GL(\Delta_n)$ is the spinor representation of $Spin^T(n)$. In case of $n=2k$ the spinor bundle splits into the sum of two subbundles $\mathbb{S}^+$ and $\mathbb{S}^-$ such that $$\mathbb{S}=\mathbb{S}^+\oplus \mathbb{S}^-, \hspace{1cm}\mathbb{S}^{\pm}=P_{Spin^T(n)}\times_{{\kappa^T}^{\pm}}\Delta_n^{\pm}.$$
\end{definition}
Any spinor field $\psi$ can be identified with the map $\psi:P_{Spin^T(n)}\rightarrow \Delta_n$ satisfying the transformation rule $\psi(pg)={\kappa^T}(g^{-1})\psi(p)$. The absolute differential of a section $\psi$ with respect to $\widetilde{\omega\times (A,B)}$ determines a covariant derivative $$\widetilde{\nabla}:\Gamma(\mathbb{S})\rightarrow \Gamma(T^*M\otimes\mathbb{S})$$ given by $$\widetilde{\nabla}\psi=d\psi+\kappa^T_{*1}(\widetilde{\omega\times (A,B)})\psi$$where $\kappa^T_{*1}:\mathfrak{spin}^T(n)\rightarrow End(\Delta_n)$ is the derivative of $\kappa$ at the identity \\ $1\in Spin^T(n)$. It can be also shown that $$\kappa^T_{*1}(e_{\alpha}e_{\beta},\lambda i,\mu i)=\kappa(e_{\alpha}e_{\beta})+(2\lambda i+\mu i)Id$$
where $\lambda$ and $\mu$ are any real numbers and $\kappa$ is the spin representation of the group Spin$(n)$.

Now we give the local formulas for connections. Fix a section $s:U\rightarrow P_{S^1\times S^1}$ of the principal bundle $P_{S^1\times S^1}$. Then, we obtain the local connection form $$(A^s,B^s):TU\rightarrow i\mathbb{R}\oplus i\mathbb{R}$$where $A^s,B^s:TU\rightarrow i\mathbb{R}$. $e\times s:U\rightarrow P_{SO(n)}\times P_{S^1\times S^1}$ is a local section of the fiber product bundle $P_{SO(n)}\times P_{S^1\times S^1}$. $\widetilde{e\times s}$ is a lift of this section to the two-fold covering $\Pi:P_{Spin^T(n)}\rightarrow P_{SO(n)}\times P_{S^1\times S^1}$. The local connection form $\widetilde{\omega\times (A,B)}^{(\widetilde{e\times s})}$ on the principal bundle $P_{Spin^T(n)}$ is given by the formula
$$\widetilde{\omega\times (A,B)}^{(\widetilde{e\times s})}=\left(\displaystyle\frac{1}{2}\sum_{i<j}g(\nabla e_i,e_j)e_{i}e_j,\frac{1}{2}A^s,B^s-\frac{1}{2}A^s\right)$$
Hence, this connection form induces a connection $\widetilde{\nabla}$ on the spinor bundle $\mathbb{S}$. We can locally describe $\widetilde{\nabla}$ by
\begin{equation}
\widetilde{\nabla}_X\psi=d\psi(X)+\displaystyle\frac{1}{2}\sum_{i<j}g(\nabla_Xe_i,e_j)e_ie_j\psi+\frac{1}{2}A^s\psi+B^s\psi
\end{equation}
where $\psi:U\rightarrow \Delta_n$ is a section of the spinor bundle $\mathbb{S}$.
\begin{definition}
The first order differential operator $$D=\mu\circ \widetilde{\nabla}:\Gamma(\mathbb{S})\overset{\widetilde{\nabla}}{\rightarrow}\Gamma(T^*M\otimes \mathbb{S})\overset{\mu}{\rightarrow} \Gamma(\mathbb{S}) $$where $\mu$ denotes Clifford multiplication, is called the Dirac operator.
\end{definition} 
The Dirac operator $D$ is locally given by
\begin{equation}
D\psi=\displaystyle\sum_{i=1}^n e_i\cdot \widetilde{\nabla}_{e_i}\psi
\end{equation}
where $\{e_1,\ldots,e_n\}$ is a local orthonormal frame on the manifold $M$.

The Dirac operator has the following property:
\begin{theorem}
Let $f$ be a smooth function and $\psi\in \Gamma(\mathbb{S})$ be a spinor
field. Then,
  $$D(f\cdot\psi)=(gradf\cdot\psi)+fD\psi.$$
\end{theorem}
\begin{proof}
By using the definition of the Dirac operator $D$ we
can compute $D(f\cdot\psi)$ as follows:
  $$\begin{array}{ccl}
    D(f\cdot\psi) & = & \overset{n}{\underset{i=1}{\sum }} e_i\cdot\widetilde{\nabla}_{e_i}(f\cdot\psi) \\
     & = & \overset{n}{\underset{i=1}{\sum }}e_i\cdot(e_i(f)\cdot\psi+f\widetilde{\nabla}_{e_i}\psi)\\
     & = & \overset{n}{\underset{i=1}{\sum }}e_i(f)e_i\cdot\psi+f\overset{n}{\underset{i=1}{\sum }}e_i\cdot\widetilde{\nabla}_{e_i}\psi\\
     & = & (grad f)\cdot\psi+fD\psi \\
  \end{array}$$
\end{proof}

Now we can define the Laplace operator on the spinor bundle $\mathbb{S}$.
\begin{definition}
  Let $\psi\in \Gamma(\mathbb{S})$ be a spinor field. The Laplace operator $\Delta$ on
  spinors is defined by
    \begin{equation}
      \Delta\psi=-\overset{n}{\underset{i=1}{\sum
      }}\left(\widetilde{\nabla}_{e_i}\widetilde{\nabla}_{e_i}\psi+div(e_i)\widetilde{\nabla}_{e_i}\psi\right).
    \end{equation}
\end{definition}

\subsection{Schrödinger-Lichnerowicz type formula}
The square $D^2$ of the Dirac operator and the Laplace operator $\Delta$ are second order differential operators. We derive Schrödinger-Lichnerowicz type formula by computing their difference $D^2-\Delta$.

The curvature $R^{\mathbb{S}}$ of the spinor covariant derivative $\widetilde{\nabla}$ is an $End(\mathbb{S})$ valued $2-$form by
$$
R^{\mathbb{S}}(X,Y)\psi=\widetilde{\nabla} _{X} \widetilde{\nabla} _{Y} \psi - \widetilde{\nabla} _{Y}\widetilde{\nabla} _{X} \psi -\widetilde{\nabla}_{[X,Y]}\psi
$$
where $\psi \in \Gamma(\mathbb{S})$ and $X,Y \in \Gamma(TM)$.
Now we want to describe $R^{\mathbb{S}}$ in terms of the curvature tensor $R$.

Let $\Omega^{\omega}:TP_{SO(n)}\times
TP_{SO(n)}\rightarrow \mathfrak{so}(n)$ be
the curvature form of the Levi-Civita connection with the
components
$$\Omega^{\omega}={\underset{i<j}{\sum}}\Omega_{ij}E_{ij}$$ where
$\Omega_{ij}:TP_{SO(n)}\times
TP_{SO(n)}\rightarrow \mathbb{R}$. The commutative
diagram defining the connection $\widetilde{\omega\times (A,B)}$ implies that the curvature
form of $\widetilde{\omega\times (A,B)}$ is
$$\Omega^{\widetilde{\omega\times (A,B)}}=\dfrac{1}{2}{\underset{i<j}{\sum}}\Pi^*(\Omega_{ij})e_ie_j\oplus \frac{1}{2}\Pi^*(dA)\oplus \Pi^*(dB).$$
Hence the $2-$form $R^{\mathbb{S}} $ with values in the spinor
bundle $\mathbb{S}$ is obtained by the following formula:
$$R^{\mathbb{S}}(.,.) \psi=\dfrac{1}{2}{\underset{i<j}{\sum}}\Omega_{ij}e_ie_j \cdot \psi+\frac{1}{2}dA\cdot \psi+dB\cdot \psi
.$$

Let $\{e_1,\ldots,e_n\}$ be
orthonormal frame field,
$\Omega_{ij}(X,Y)=g(R(X,Y)e_i,e_j)$ the
components of the curvature form of the Levi-Civita connection,\\
$X=\overset{n}{\underset{k=1}{\sum}}X^k e_k$ and
$Y=\overset{n}{\underset{l=1}{\sum}} Y^l e_l$ be vector fields on
the Riemannian manifold $M$. Then we have
$$\begin{array}{ccl}
  \Omega_{ij}(X,Y) & = & g(R(X,Y)e_i,e_j) \\
   & = & \overset{n}{\underset{k,l=1}{\sum}}R_{klij}X^kY^l \\
   & = & \overset{n}{\underset{k,l=1}{\sum}}R_{klij}e^k(X)e^l(Y) \\
   & = & \dfrac{1}{2}\overset{n}{\underset{k,l=1}{\sum}}R_{klij}(e^k\wedge e^l)(X,Y). \\
\end{array}$$where $\{e^1, \ldots,e^{n}\}$ is the frame dual to
$\{e_1,\ldots,e_{n}\}$. Thus, we obtain the following local formula
for the curvature form
$$\Omega^{\widetilde{\omega\times (A,B)}}=\dfrac{1}{4}\overset{}{\underset{i<j}{\sum}}\overset{n}{\underset{k,l=1}{\sum}}R_{klij}(e^k\wedge e^l)e_ie_j+\frac{1}{2}dA+ dB
$$
and the 2-form $R^{\mathbb{S}}(.,.)$ is calculated as follows:
$$R^{\mathbb{S}}(.,.)\psi=\dfrac{1}{4}\overset{}{\underset{i<j}{\sum}}\overset{n}{\underset{k,l=1}{\sum}}R_{klij}(e^k\wedge e^l)e_ie_j\cdot\psi+\frac{1}{2}dA\cdot \psi+ dB\cdot \psi
. $$

By using the above properties of the curvature form $R^{\mathbb{S}}$ on
spinor bundle $\mathbb{S}$ we deduce the following result:
\begin{proposition}
Let $Ric$ be the Ricci tensor. Then, the
following relation holds:

\begin{equation}
\label{36}
\overset{n}{\underset{\alpha=1}{\sum}}e_{\alpha}\cdot R^{\mathbb{S}}(X,e_{\alpha})\psi=-\dfrac{1}{2}Ric(X)\cdot \psi+\frac{1}{2}(X\mathrel{\lrcorner} dA)\cdot \psi+(X\mathrel{\lrcorner}dB)\cdot \psi
\end{equation}
\end{proposition}

\begin{proof}
In \cite{F} it is proved the following relation:
\begin{equation}\label{31}
\overset{n}{\underset{\alpha=1}{\sum}}\overset{}{\underset{i<j}{\sum}}\overset{n}{\underset{k,l=1}{\sum}}R_{klij}(e^k\wedge e^l)e_{\alpha}e_ie_j\cdot\psi=-2Ric(X)\cdot \psi
\end{equation}
It can be easily seen the following two relations:
\begin{equation}\label{32}
\displaystyle\sum_{\alpha=1}^n e_{\alpha}\cdot dA(X,e_{\alpha})\cdot \psi=(X\mathrel{\lrcorner} dA)\cdot \psi
\end{equation}
and
\begin{equation}\label{33}
\displaystyle\sum_{\alpha=1}^n e_{\alpha}\cdot dB(X,e_{\alpha})\cdot \psi=(X\mathrel{\lrcorner} dB)\cdot \psi.
\end{equation}
Then, using (\ref{31}), (\ref{32}) and (\ref{33}), we obtain the claimed equivalence.
\end{proof}

Now, we derive Schrödinger-Lichnerowicz-type formula in the following way:

\begin{proposition}
Let $s$ be scalar curvature of the Riemannian manifold and let $dA=\Omega^A$ and $dB=\Omega^B$ be the imaginary-valued $2-$forms of the connections $(A,B)$ in the $(S^1\times S^1)-$bundle associated with Spin$^T$ structure. Then, we have the following formula:
$$D^2\psi=\Delta \psi+\frac{s}{4}\psi+\frac{1}{2}dA\cdot \psi+dB\cdot \psi.$$
\end{proposition}

\begin{proof}
\begin{equation}\label{34}
\begin{array}{lll}
     D^2\psi & = &\displaystyle \sum_{i,j}e_i\cdot \widetilde{\nabla}_{e_i}(e_j\cdot \widetilde{\nabla}_{e_j}\psi) \\
      & =  & \displaystyle \sum_{i,j}e_i\cdot \nabla_{e_i}e_j\cdot \widetilde{\nabla}_{e_j}\psi+ e_ie_j\cdot \widetilde{\nabla}_{e_i}\widetilde{\nabla}_{e_j}\psi \\
      & =  & \displaystyle \sum_{i,j,k}g(\nabla_{e_i}e_j,e_k) e_ie_k\cdot \widetilde{\nabla}_{e_j}\psi+ \sum_{i,j}e_ie_j\cdot \widetilde{\nabla}_{e_i}\widetilde{\nabla}_{e_j}\psi  \\ 
      & =  & \Delta \psi+\displaystyle \sum_{j,i\neq k}g(\nabla_{e_i}e_j,e_k) e_ie_k\cdot\widetilde{\nabla}_{e_j}\psi+ \sum_{i\neq j}e_ie_j\cdot \widetilde{\nabla}_{e_i}\widetilde{\nabla}_{e_j}\psi\\
\end{array}
\end{equation}
Now we can calculate the following sum:
\begin{equation*}
\begin{array}{lll}
   \displaystyle \sum_{i\neq k}g(\nabla_{e_i}e_j,e_k) e_ie_k & = & -\displaystyle \sum_{i\neq k}g(e_j,\nabla_{e_i}e_k) e_ie_k\\
   &=&  -\displaystyle \sum_{i< k}g(e_j,\nabla_{e_i}e_k-\nabla_{e_k}e_i) e_ie_k\\
      &=& \displaystyle \sum_{i< k}g(e_j,[e_k,e_i]) e_ie_k\\
   \end{array}
\end{equation*}
From (\ref{34}) we get
\begin{equation*}
\begin{array}{lll}
     D^2\psi & =  & \Delta \psi+\displaystyle \sum_{j,i< k}g(e_j,[e_k,e_i]) e_ie_k\widetilde{\nabla}_{e_j}\psi+ \sum_{i<j}e_ie_j\cdot (\widetilde{\nabla}_{e_i}\widetilde{\nabla}_{e_j}\psi-\widetilde{\nabla}_{e_j}\widetilde{\nabla}_{e_i}\psi)\\
     &=&\Delta \psi+\displaystyle \sum_{i< j}e_ie_j(\widetilde{\nabla}_{e_i}\widetilde{\nabla}_{e_j}\psi-\widetilde{\nabla}_{e_j}\widetilde{\nabla}_{e_i}\psi-\widetilde{\nabla}_{[e_i,e_j]}\psi)\\
          &=&\Delta \psi+\displaystyle\frac{1}{2}\sum_{i, j}e_ie_jR^{\mathbb{S}}(e_i,e_j)\psi.\\
\end{array}
\end{equation*}
Using the identity (\ref{36}) and multiplying by $e_i$ we deduce

\begin{equation*}
\begin{array}{lll}
     D^2\psi &=&\Delta \psi-\displaystyle\frac{1}{4}\sum_{i}e_iRic(e_i)\cdot \psi+\frac{1}{4}\sum_{i}e_i\cdot(e_i \mathrel{\lrcorner}dA)\cdot \psi+\frac{1}{2}\sum_{i}e_i\cdot(e_i \mathrel{\lrcorner}dB)\cdot \psi\\
 &=&\Delta \psi+\displaystyle\frac{s}{4}\psi+\displaystyle\frac{1}{2}dA\cdot \psi+dB\cdot \psi.
\end{array}
\end{equation*}

\end{proof}

\end{document}